\documentclass{amsart}

\usepackage{amssymb}
\usepackage{mathrsfs}

\theoremstyle{plain}
\newtheorem{theorem}{Theorem}[section]
\newtheorem{lemma}[theorem]{Lemma}
\newtheorem{corollary}[theorem]{Corollary}

\theoremstyle{remark}
\newtheorem*{remark}{Remark}
\newtheorem*{remarks}{Remarks}

\numberwithin{equation}{section}
\numberwithin{table}{section}

\newcommand{\ZpL}[1]{\mathbb{Z}_{p}^{\mathcal{L}_{{#1}}}}

\newcommand{\FqL}[1]{\mathbb{F}_{q}^{\mathcal{L}_{{#1}}}}

\newcommand{\RL}[1]{R^{\mathcal{L}_{{#1}}}}
\newcommand{\FL}[1]{F^{\mathcal{L}_{{#1}}}}
\newcommand{\BB}{{\mathcal{B}}}
\newcommand{\CC}{{\mathcal{C}}}
\newcommand{\DD}{{\mathcal{D}}}
\newcommand{\RR}{{R}}
\newcommand{\LL}{{\mathcal{L}}}
\newcommand{\HH}{{\mathcal{H}}}
\newcommand{\Zp}{{\mathbb{Z}_{p}}}
\newcommand{\Qp}{{\mathbb{Q}_{p}}}
\newcommand{\Fp}{{\mathbb{F}_{p}}}
\newcommand{\Fq}{{\mathbb{F}_{q}}}

\newcommand{\leftsub}[2]{{\vphantom{#2}}_{#1}{#2}}
\newcommand{\qbinom}[2]{{\left[\begin{smallmatrix}#1\\#2\end{smallmatrix}\right]_q}}
\newcommand{\abs}[1]{\lvert#1\rvert}
\DeclareMathOperator{\im}{Im}
\DeclareMathOperator{\Hom}{{\mathrm{Hom}}}
\DeclareMathOperator{\GL}{{\mathrm{GL}}}
\DeclareMathOperator{\PG}{{\mathrm{PG}}}

\begin{document}

\title[Elementary Divisors]{The Elementary Divisors of the Incidence Matrix of Skew Lines in $\mathrm{PG}(3,q)$}

\author{Andries E.  Brouwer}
\address{Dept. of Mathematics, Techn. Univ. Eindhoven, 5600MB Eindhoven, Netherlands} 
\email{aeb@cwi.nl}

\author{Joshua E. Ducey}
\address{Dept. of Mathematics, University of Florida, Gainesville, FL 32611--8105, USA}
\email{jducey21@ufl.edu}

\author{Peter Sin}
\address{Dept. of Mathematics, University of Florida, Gainesville, FL 32611--8105, USA}
\email{sin@ufl.edu}

\subjclass[2010]{Primary 05B20, 05E30; Secondary 20C33, 51E20}

\begin{abstract}
The elementary divisors of the incidence matrix of lines in $\operatorname{PG}(3,q)$ are computed, where two lines are incident if and only if they are skew.
\end{abstract}

\maketitle

%
%
%
%
%
\section{Introduction}
Let $V$ be a $4$-dimensional vector space over the finite field $\mathbb{F}_{q}$ of $q=p^t$ elements, where $p$ is a prime.  We declare two $2$-dimensional subspaces $U$ and $W$ to be incident if and only if $U \cap W = \{0\}$.  Ordering the $2$-dimensional subspaces in some arbitrary but fixed manner, we can form the incidence matrix $A$ of this relation.  By the well-known Klein correspondence, $A$ is also the adjacency matrix of the non-collinearity graph on the points of the Klein quadric.  The entries of this zero-one matrix may be read over any commutative ring.  In this paper we compute the Smith normal form of $A$ as an integer matrix.

In order to introduce some useful notation, we will view this setup as a special case of a more general situation.  For brevity, an $r$-dimensional subspace will be called an $r$-subspace in what follows.

More generally, let $V$ be an $(n+1)$-dimensional vector space over the finite field $\mathbb{F}_{q}$, where $q = p^{t}$ is a prime power.  Let $\mathcal{L}_{r}$ denote the set of $r$-subspaces of $V$.  Thus $\mathcal{L}_{1}$ denotes the points, $\mathcal{L}_{2}$ denotes the lines, etc. in $\mathbb{P}(V)$.  An $r$-subspace $U \in \mathcal{L}_{r}$ and an $s$-subspace $W \in \mathcal{L}_{s}$ are incident if and only if $U \cap W = \{0\}$.  The incidence matrix with rows indexed by the $r$-subspaces and columns indexed by the $s$-subspaces is denoted $A_{r,s}$.

These matrices $A_{r,s}$ and other closely related ones have been studied by a number of mathematicians.  The reader is referred to the surveys~\cite{survey1, survey2} (see also~\cite[Introduction]{chandler:sin:xiang:2006}).  Their $p$-ranks are determined in~\cite{sin:2004}, but their elementary divisors have been computed only in the case that either $r$ or $s$ is equal to one~\cite{lander, sin:2000, chandler:sin:xiang:2006}.

We return now to the situation where $V$ is $4$-dimensional over $\mathbb{F}_{q}$ and $A = A_{2,2}$.
\section{The Main Results}
It turns out that the elementary divisors of $A$ are all powers of $p$.  A quick way to see this is to regard $A$ as the adjacency matrix of the graph with vertex set $\mathcal{L}_{2}$, where two lines are adjacent when skew.  This is a strongly regular graph, with parameters $v = q^4 + q^3 + 2q^2 + q + 1$, $k = q^4$, $\lambda = q^4 - q^3 - q^2 + q$, $\mu = q^4 - q^3$.  Thus $A$ satisfies the equation
\begin{equation}
\label{eq:matrix1}
A^2 = q^{4}I + (q^4 - q^3 - q^2 + q)A + (q^4 - q^3)(J - A - I),
\end{equation}
where $I$ and $J$ denote the identity matrix and all-one matrix, respectively, of the appropriate sizes.  From this equation one deduces that the eigenvalues of $A$ are $q$, $-q^2$, and $q^4$ with respective multiplicities $q^4 + q^2$, $q^3 + q^2 + q$, and $1$.  Since $\abs{\det(A)}$ is the product of the elementary divisors, we see that the elementary divisors of $A$ are all powers of $p$.

The following two theorems give the Smith normal form of $A$.
\begin{theorem}
\label{thm:A}
Let $e_{i} = e_{i}(A)$ denote the multiplicity of $p^i$ as an elementary divisor of $A$.
\begin{enumerate}
\item \label{item:A1} $e_{i} = e_{3t-i}$ for $0 \le i < t$.
\item \label{item:A2} $e_{i} = 0$ for $t < i < 2t$, $3t < i < 4t$, and $i > 4t$.
\item \label{item:A3} $\sum_{i=0}^{t}e_{i} = q^4 + q^2$.
\item \label{item:A4} $\sum_{i=2t}^{3t}e_{i} = q^3 + q^2 + q$.
\item \label{item:A5} $e_{4t} = 1$.
\end{enumerate}
\end{theorem}
Thus we get all the elementary divisor multiplicities once we know $t$ of the numbers $e_{0}, \dots , e_{t}$ (or the numbers $e_{2t}, \dots ,e_{3t}$).  The next theorem describes these.  To state the theorem, we need some notation.

Set
\[
[3]^{t} = \{(s_0, \dots ,s_{t-1}) \, | \, s_{i} \in \{1,2,3\} \hbox{ for all $i$}\}
\]
and
\[
\mathcal{H}(i) = \bigl\{(s_0, \dots ,s_{t-1}) \in [3]^{t} \, \big\vert \, \#\{j|s_{j}=2\}=i\bigr\}.
\]
In other words, $\mathcal{H}(i)$ consists of the tuples in $[3]^{t}$ with exactly $i$ twos.  To each tuple $\vec{s} \in [3]^{t}$ we associate a number $d(\vec{s})$ as follows.  For $\vec{s} = (s_0, \dots ,s_{t-1}) \in [3]^{t}$ define the integer tuple $\vec{\lambda} = (\lambda_{0}, \dots ,\lambda_{t-1})$ by 
\[
\lambda_{i} = ps_{i+1} - s_{i},
\]
with the subscripts read modulo %
%
%
$t$. For an integer $k$, set $d_{k}$ to be the coefficient of $x^{k}$ in the expansion of $(1 + x + \cdots + x^{p-1})^{4}$.  Finally, set $d(\vec{s}) = \prod_{i=0}^{t-1}d_{\lambda_{i}}$.
\begin{theorem}
\label{thm:B}
Let $e_{i} = e_{i}(A)$ denote the multiplicity of $p^i$ as an elementary divisor of $A$.  Then, for $0 \le i \le t$,
\[
e_{2t+i} = \sum_{\vec{s} \in \mathcal{H}(i)}d(\vec{s}).
\]
\end{theorem}

\begin{remark}
When $p=2$, notice that $d(\vec{s})=0$ for any tuple $\vec{s}$ containing an adjacent $1$ and $3$.  Also, $d(\vec{s})=0$ for tuples $\vec{s}$ beginning with a $1$ and ending with a $3$ (and vice versa).  %
%
%
Thus the sum in Theorem~\ref{thm:B} is significantly easier to compute in this case. 
\end{remark}

As an example of how to use the above theorems, consider the case when $p=3$, $t=2$.  We have
\[
(1+x+x^{2})^{4} = 1 + 4x + 10x^{2} + 16x^{3} + 19x^{4} + 16x^{5} + 10x^{6} + 4x^{7} + x^{8},
\]
\begin{align*}
\HH(0) &= \{(11), (13), (31), (33)\}, \\
\HH(1) &= \{(21), (23), (12), (32)\}, \\
\HH(2) &= \{(22)\}.
\end{align*}
Using Theorem~\ref{thm:B} we compute
\begin{align*}
e_{4} &= d(11) + d(13) + d(31) + d(33) \\
      &= d_{2} \cdot d_{2} + d_{8} \cdot d_{0} + d_{0} \cdot d_{8} + d_{6} \cdot d_{6} \\
      &= 10 \cdot 10 + 1 \cdot 1 + 1 \cdot 1 + 10 \cdot 10 \\
      &= 202,
\end{align*}
\begin{align*}
e_{5} &= d(21) + d(23) + d(12) + d(32) \\
      &= d_{1} \cdot d_{5} + d_{7} \cdot d_{3} + d_{5} \cdot d_{1} + d_{3} \cdot d_{7} \\
      &= 4 \cdot 16 + 4 \cdot 16 + 16 \cdot 4 + 16 \cdot 4 \\
      &= 256,
\end{align*}
\[
e_{6} = d(22) = d_{4} \cdot d_{4} = 19 \cdot 19 = 361.
\]
The remaining nonzero multiplicities are now given by Theorem~\ref{thm:A}.  We collect this information in Table~\ref{table:q=9}.
\begin{table}[hdtp]
\begin{center}
\newlength{\Eheight}
\newlength{\strutheight}
\settoheight{\Eheight}{E}
\setlength\strutheight{1.5\Eheight}
\caption{The elementary divisors of the incidence matrix of lines vs.\ lines in $\PG(3,9)$, where two lines are incident when skew.}
\label{table:q=9}
\begin{tabular}{lrrrrrrr}
\hline
\rule{0pt}{\strutheight}Elem. Div. & $1$ & $3$ & $3^2$ & $3^4$ & $3^5$ & $3^6$ & $3^8$ \\
\rule{0pt}{\strutheight}Multiplicity & $361$ & $256$ & $6025$ & $202$ & $256$ & $361$ & $1$ \\ \hline
\end{tabular}
\end{center} 
\end{table}
\section{Elementary Divisors and Smith Normal Form Bases}%
%
%
In this section we collect a few useful results regarding elementary divisors.  Let $\RR$ be a discrete valuation ring, %
%
%
with $p \in \RR$ a prime generating the maximal ideal.  An $m \times n$ matrix with entries in $\RR$ can be viewed as a homomorphism of free $\RR$-modules of finite rank:
\[
\eta\colon \RR^{m} \to \RR^{n}.
\]
The elementary divisors of $\eta$ are by definition just the elementary divisors of the matrix, and for a fixed prime $p$ we always let $e_{i}(\eta)$ denote the multiplicity of $p^{i}$ as an elementary divisor of $\eta$.

Set $F = \RR/p\RR$.  If $L$ is an $\RR$-submodule of a free $\RR$-module $\RR^{l}$, then $\overline{L} = (L + p\RR^{l})/p\RR^{l}$ is an $F$-vector space.  For $i \ge 0$, define
\[
M_{i}(\eta) = \{x \in \RR^m \, | \, \eta(x) \in p^{i}\RR^n\}
\]
and
\[
N_{i}(\eta) = \{p^{-i}\eta(x) \, | \, x \in M_{i}(\eta)\}.
\]
For convenience we also define $N_{-1}(\eta) = \{0\}$.  Then we have chains of $\RR$-modules
\[
\RR^{m}=M_{0}(\eta) \supseteq M_{1}(\eta) \supseteq \cdots
\]
\[
N_{0}(\eta) \subseteq N_{1}(\eta) \subseteq \cdots
\]
and chains of $F$-vector spaces
\[
F^{m}=\overline{M_{0}(\eta)} \supseteq \overline{M_{1}(\eta)} \supseteq \cdots
\]
\[
\overline{N_{0}(\eta)} \subseteq \overline{N_{1}(\eta)} \subseteq \cdots.
\]
\begin{lemma}
\label{lem:A}
Let $\eta\colon \RR^m \to \RR^n$ be a homomorphism of free $\RR$-modules of finite rank, and let $e_{i}(\eta)$ denote the multiplicity of $p^{i}$ as an elementary divisor of $\eta$.  Then, for $i \ge 0$,
\[
e_{i}(\eta) = \dim_{F}\left(\overline{M_{i}(\eta)}/\overline{M_{i+1}(\eta)}\right) = \dim_{F}\left(\overline{N_{i}(\eta)}/\overline{N_{i-1}(\eta)}\right).
\]
\end{lemma}
\begin{proof}
From the theory of modules over principal ideal domains, there exists a basis $\BB$ of $\RR^{m}$ and a basis $\CC$ of $\RR^{n}$ with respect to which the matrix of $\eta$ is in Smith normal form.  Each of the above $\RR$-submodules $M_{i}(\eta)$ (resp.\ $N_{j}(\eta)$) are seen to have a basis consisting of $p$-power multiples of elements of $\BB$ (resp.\ $\CC$).  When we represent %
%
%
the modules in this way, the lemma is clear.
\end{proof}

For a given homomorphism $\eta\colon \RR^m \to \RR^{n}$, we will be interested in pairs of bases ($\BB$, $\CC$) with respect to which the matrix of $\eta$ is in diagonal form.  %
%
%
We define a \textit{left} SNF basis for $\eta$ to be any basis of $\RR^{m}$ that belongs to such a pair.  Similarly, a \textit{right} SNF basis for $\eta$ is any basis of $\RR^{n}$ belonging to such a pair.  We now describe how to construct such bases.

Suppose $\eta\colon \RR^{m} \to \RR^{n}$ is nonzero.  Then there is a unique largest nonnegative integer $l$ with $e_{l}(\eta) \neq 0$.  We have 
\[
\overline{M_{0}(\eta)} \supseteq \overline{M_{1}(\eta)} \supseteq \cdots \supseteq \overline{M_{l}(\eta)} \supsetneq \overline{\ker(\eta)}.
\]
where only the last inclusion is necessarily strict.  Choose a basis $\overline{\BB_{l+1}}$ of $\overline{\ker(\eta)}$ and extend it to a basis $\overline{\BB_{l}} \cup \overline{\BB_{l+1}}$ of $\overline{M_{l}(\eta)}$.  Continue in this fashion to get a basis $\cup_{i = 0}^{l+1}\overline{\BB_{i}}$ of $\overline{M_{0}(\eta)}$.  Now lift the elements of $\overline{\BB_{l+1}}$ to a set $\BB_{l+1}$ of preimages in $\ker(\eta)$.  Continuing, at each stage we enlarge $\BB_{i+1}$ by adjoining a set $\BB_{i}$ of preimages in $M_{i}(\eta)$ of $\overline{\BB_{i}}$.  By Nakayama's Lemma, the set
\[
\BB = \bigcup_{i = 0}^{l+1} \BB_{i}
\]
is an $\RR$-basis of $R^{m}$.

Notice that $N_{l}(\eta) = N_{l+1}(\eta) = \cdots$ .  Set $N' = N_{l}(\eta)$.  Then $N'$ is called the \textit{purification} of $\im \eta$, and is the smallest $\RR$-module direct summand of $\RR^{n}$ containing $\im \eta$.  The elementary divisors of $\eta$ remain the same if we change the codomain of $\eta$ to $N'$.  Choose a basis $\overline{\CC_{0}}$ of $\overline{N_{0}(\eta)}$ and extend it to a basis $\overline{\CC_{0}} \cup \overline{\CC_{1}}$ of $\overline{N_{1}(\eta)}$.  Continue in this fashion to get a basis $\cup_{i = 0}^{l} \overline{\CC_{i}}$ of $\overline{N_{l}(\eta)}$.  Now we lift the elements of $\overline{\CC_{0}}$ to a set $\CC_{0}$ of preimages in $N_{0}(\eta)$.  Continuing, at each stage we enlarge $\CC_{i}$ by adjoining a set $\CC_{i+1}$ of preimages in $N_{i+1}(\eta)$ of $\overline{\CC_{i+1}}$.  By Nakayama's Lemma, the set 
\[
\CC' = \bigcup_{i = 0}^{l} \CC_{i}
\]
is an $\RR$-basis of $N'$.  We then set 
\[
\CC = \bigcup_{i = 0}^{l+1} \CC_{i}
\]
to be any $\RR$-basis of $\RR^{n}$ obtained by adjoining to $\CC'$ some set $\CC_{l+1}$.
\begin{lemma}
\label{lemma:E}
\hfil
\begin{enumerate}
\item \label{item:E1} The basis $\BB$ constructed above is a left SNF basis for $\eta$. \\
\item \label{item:E2} The basis $\CC$ constructed above is a right SNF basis for $\eta$.
\end{enumerate}
\end{lemma}
\begin{proof}
For $x \in \BB_{i}$, $0 \le i \le l$, consider the element $y = p^{-i}\eta(x) \in N'$.  The collection of all such elements form a linearly independent set, since the basis $\BB$ extends the basis $\BB_{l+1}$ of $\ker(\eta)$.  Let $Y$ denote the $\RR$-submodule generated by these elements.  From Lemma~\ref{lem:A} we see that the index of $\im \eta$ in $Y$ is the same as the index of $\im \eta$ in $N'$.  Hence $Y = N'$, and so these elements form a basis of $N'$.  The matrix of $\eta$ with respect to $\BB$ and any basis of $\RR^{n}$ obtained by extending this basis of $N'$ will then be in diagonal form.  This proves part~\eqref{item:E1}.

Now, for each $y \in \CC_{i}$, $0 \le i \le l$, choose an element $x \in M_{i}(\eta)$ such that $\eta(x) = p^{i}y$.  Let $X$ denote the $\RR$-submodule of $\RR^{m}$ generated by these elements.  The images of these elements are certainly linearly independent, hence $X \cap \ker(\eta) = \{0\}$.  By Lemma~\ref{lem:A} we see that $\im \eta$ and $\eta(X + \ker(\eta))$ have the same index in $N'$.  Therefore $\RR^{m} = X \oplus \ker(\eta)$, and adjoining any basis of $\ker(\eta)$ to these generators of $X$ gives a basis of $\RR^{m}$.  With respect to this basis and $\CC$, the matrix of $\eta$ is in diagonal form.
\end{proof}
We end this section with an easy but very useful result.
\begin{lemma}
\label{lem:B}
Let $\gamma\colon \RR^m \to \RR^n$ be another $\RR$-module homomorphism, and suppose that for some $k \ge 1$ we have
\[
\eta(x) \equiv \gamma(x) \pmod {p^{k}}, \hbox{\quad for all $x \in \RR^m$}.
\]
Then
\[
e_{i}(\eta) = e_{i}(\gamma), \hbox{\quad for $0 \le i \le k-1$}.
\]
\end{lemma}
\begin{proof}
Verify that $M_{i}(\eta) = M_{i}(\gamma)$,  for $0 \le i \le k$.  The conclusion is now immediate from Lemma~\ref{lem:A}.
\end{proof}
\section{Proof of Theorem~\ref{thm:A}}
Since all of the elementary divisors of $A$ are powers of $p$, we lose nothing by viewing %
%
%
$A$ as a matrix over $\Zp$, the ring of $p$-adic integers.  The matrix $A$ %
%
%
represents a homomorphism of free $\Zp$-modules 
\[
\ZpL{2} \to \ZpL{2}
\]
that sends a $2$-subspace to the (formal) sum of the $2$-subspaces incident with it.  We abuse notation by using the same symbol for both the matrix and the map.  We also apply our matrices and maps on the right (so $AB$ means ``do $A$ first, then $B$'').

Let $\mathbf{1} = \sum_{x \in \mathcal{L}_{2}}x$ and set
\[
Y_{2} = \Bigl\{\sum_{x \in \mathcal{L}_{2}}a_{x}x \in \ZpL{2} \, \Big\vert \, \sum_{x \in \mathcal{L}_{2}}a_{x} = 0\Bigr\}.
\]
Since $\abs{\mathcal{L}_{2}}$ is a unit in $\Zp$, we have the decomposition
\begin{equation*}
\ZpL{2} = \Zp\mathbf{1} \oplus Y_{2}.
\end{equation*}

We now prove Theorem~\ref{thm:A}.  The map $A$ respects the above decomposition of $\ZpL{2}$, and thus we get all of the elementary divisors of $A$ by computing those of the restriction of $A$ to each summand.  Since $(\mathbf{1})A = q^{4}\mathbf{1}$, we see that $e_{4t}(A) = e_{4t}(A|_{Y_{2}}) + 1$ and $e_{i}(A) = e_{i}(A|_{Y_{2}})$ for $i \neq 4t$.

Rewriting equation~\eqref{eq:matrix1} we get
\[
A(A + (q^{2}-q)I) = q^{3}I + (q^4 - q^3)J
\]
and if we now restrict $A$ to $Y_{2}$, the above equation reads%
%
%
\begin{equation}
\label{eq:matrix2}
A|_{Y_{2}}(A|_{Y_{2}} + (q^{2}-q)I) = q^{3}I.
\end{equation}
Looking carefully at this equation one sees that a left (resp. right) SNF basis for $A|_{Y_{2}}$ is a right (resp. left) SNF basis for $A|_{Y_{2}} + (q^{2}-q)I$.  It follows from equation~\eqref{eq:matrix2} that $e_{i}(A|_{Y_{2}})=0$ for $i > 3t$, and so $e_{4t}(A)=1$.  Thus we have established part~\eqref{item:A5} of the theorem (and most of part~\eqref{item:A2}).%
%
%

It also follows immediately from equation~\eqref{eq:matrix2} that
\begin{equation}
\label{eq:sym}
e_{i}(A|_{Y_{2}} + (q^{2}-q)I) = e_{3t-i}(A|_{Y_{2}})
\end{equation}
for $0 \le i \le 3t$.  Since $A|_{Y_{2}} \equiv A|_{Y_{2}} + (q^{2}-q)I \pmod {p^{t}}$, we have from Lemma~\ref{lem:B} that
\begin{equation}
e_{i}(A|_{Y_{2}}) = e_{i}(A|_{Y_{2}} + (q^{2}-q)I) = e_{3t-i}(A|_{Y_{2}})
\end{equation}
for $0 \le i < t$, which is part~\eqref{item:A1} of the theorem.

It remains to prove parts~\eqref{item:A3} and~\eqref{item:A4} of the theorem, and also the statement from part~\eqref{item:A2} that $e_{i}(A) = 0$ for $t < i < 2t$.  Denote by $V_{\lambda}$ the $\lambda$-eigenspace for $A$ (as a matrix over $\Qp$, the $p$-adic numbers).  Since $V_{q}$ is a $\Qp$-subspace of $\Qp^{\LL_{2}}$, the intersection $V_{q} \cap \ZpL{2}$ is a pure sublattice of $\ZpL{2}$ and the same is true for $V_{-q^2} \cap \ZpL{2}$.  %
%
%
Notice that $V_{q} \cap \ZpL{2} \subseteq N_{t}(A|_{Y_{2}})$ and $V_{-q^2} \cap \ZpL{2} \subseteq M_{2t}(A|_{Y_{2}})$.  Therefore,
\[
q^4 + q^2 = \dim_{\Fp}(\overline{V_{q} \cap \ZpL{2}}) \le \dim_{\Fp}\overline{N_{t}(A|_{Y_{2}})} = \sum_{i=0}^{t}e_{i}(A|_{Y_{2}})
\]
and
\[
q^3 + q^2 + q = \dim_{\Fp}(\overline{V_{-q^2} \cap \ZpL{2}}) \le \dim_{\Fp}\overline{M_{2t}(A|_{Y_{2}})} = \sum_{i=2t}^{3t}e_{i}(A|_{Y_{2}}).
\]
Since $(q^4 + q^2) + (q^3+q^2+q) = \dim_{\Fp}\overline{Y_{2}}$, the above inequalities are actually \textit{equalities}, and the remaining elementary divisor multiplicities must be zero.  This completes the proof of Theorem~\ref{thm:A}. \qed

\begin{remark}
The above proof simply exploits equation~\eqref{eq:matrix1}, and makes no use of the geometry of $\PG(3,q)$.  Therefore Theorem~\ref{thm:A} is also true for the adjacency matrix $A$ of any strongly regular graph with the same parameters.
\end{remark}

Theorem~\ref{thm:B} will follow from a more general result which we prove below.  Here we explain the connection between these theorems.  Let $B$ denote the incidence matrix with rows indexed by $\LL_{1}$ and columns indexed by $\LL_{2}$, where incidence again means zero intersection.  $B^t$ denotes the transpose of $B$, and is just the incidence matrix of lines vs\@. points.  It is easy to check that
\begin{equation}
\label{eq:matrixB}
B^{t}B = (q^3 + q^2)I + (q^3 + q^2 - q - 1)A + (q^3 + q^2 -q)(J - A - I).
\end{equation}
Just like with $A$, we denote also by $B$ and $B^t$ the incidence maps these matrices represent over $\Zp$.  Notice that $(\mathbf{1})B^{t}B = q^{4}(q^2 + q + 1)(q + 1)\mathbf{1}$, and so for $i \neq 4t$ we have $e_{i}(B^{t}B) = e_{i}(B^{t}B|_{Y_{2}})$.  Thus again we concentrate on the summand $Y_{2}$.

We can rewrite the equation~\eqref{eq:matrixB} as 
\[
B^{t}B = -[A + (q^2 - q)I] + q^{2}I + (q^3 + q^2 - q)J
\]
and upon restriction of maps to $Y_{2}$ it reads 
\[
B^{t}B|_{Y_{2}} = -[A|_{Y_{2}} + (q^2 - q)I] + q^{2}I.
\]%
%
%
Applying Lemma~\ref{lem:B} we have, for $0 \le i < 2t$, 
\begin{equation}
e_{i}(B^{t}B|_{Y_{2}}) = e_{i}(A|_{Y_{2}} + (q^2 - q)I).
\end{equation}
Using equation~\eqref{eq:sym}, and considering only nonzero multiplicities, we then get the relations%
%
%
\begin{equation}
\label{eq:mult}
e_{2t+i}(A) = e_{t-i}(B^{t}B), \hbox{\quad for $0 \le i \le t$.}
\end{equation}

Therefore to prove Theorem~\ref{thm:B} it is sufficient to compute the ($p$-adic) elementary divisors %
%
%
of the matrix $B^{t}B$.  The final theorem below describes these.  We can actually do this at the level of generality mentioned in the introduction.
\section{The General Result}
For the remainder of the paper, $V$ is an $(n+1)$-dimensional vector space over $\Fq$, where $q = p^t$ is a prime power.  $A_{r,s}$ is the $\abs{\LL_{r}} \times \abs{\LL_{s}}$ incidence matrix with rows indexed by the $r$-subspaces of $V$ and columns indexed by the $s$-subspaces of $V$, and two subspaces are incident if and only if their intersection is trivial.  We will compute the elementary divisors of $A_{r,1}A_{1,s}$ as a matrix over $\Zp$.%
%
%

Let $\HH$ denote the set of $t$-tuples of integers $\vec{s}=(s_{0}, \dots , s_{t-1})$ that satisfy, for $0\le i\le t-1$,
\begin{enumerate}
\item $1 \le s_{i} \le n$,
\item $0 \le ps_{i + 1} - s_{i} \le (p-1)(n+1)$,
\end{enumerate}
with subscripts read modulo $t$.  First introduced in~\cite{hamada}, the set $\HH$ was later used in~\cite{bardoe:sin} to describe the module structure of $\FqL{1}$ under the action of $\GL(n+1,q)$.  For nonnegative integers $\alpha, \beta$, define the subsets of $\HH$
\[
\HH_{\alpha}(s) = \Bigl\{(s_{0}, \dots , s_{t-1}) \in \HH \, \Big\vert \, \sum_{i=0}^{t-1}\max\{0, s-s_{i}\} = \alpha\Bigr\}
\]
and
\begin{align*}
\leftsub{\beta}{\HH}(r) &= \{(n+1-s_{0}, \dots , n+1-s_{t-1}) \, | \, (s_{0}, \dots , s_{t-1}) \in \HH_{\beta}(r)\} \\
&= \Bigl\{(s_{0}, \dots , s_{t-1}) \in \HH \, \Big\vert \, \sum_{i=0}^{t-1} \max\{0, s_{i}-(n+1-r)\} = \beta\}\Bigr\}.
\end{align*}

To each tuple $\vec{s} \in \HH$ we associate a number $d(\vec{s})$ as follows.  For $\vec{s} = (s_0, \dots ,s_{t-1}) \in \HH$ define the integer tuple $\vec{\lambda} = (\lambda_{0}, \dots ,\lambda_{t-1})$ by 
\[
\lambda_{i} = ps_{i+1} - s_{i} \hbox{\quad (subscripts mod $t$)}.
\]
For an integer $k$, set $d_{k}$ to be the coefficient of $x^{k}$ in the expansion of $(1 + x + \cdots + x^{p-1})^{n+1}$.  Explicitly,
\[
d_{k} = \sum_{j = 0}^{\lfloor k/p \rfloor}(-1)^{j}\binom{n+1}{j}\binom{n+k-jp}{n}.
\]
Finally, set $d(\vec{s}) = \prod_{i=0}^{t-1}d_{\lambda_{i}}$.
\begin{theorem}
\label{thm:C}
Let $e_{i}=e_{i}(A_{r,1}A_{1,s})$ denote the multiplicity of $p^{i}$ as a $p$-adic elementary divisor of $A_{r,1}A_{1,s}$.%
%
%
\begin{enumerate}
\item $e_{t(r+s)} = 1$.
\item For $i \neq t(r+s)$,
\[
e_{i} = \sum_{\vec{s} \in \Gamma(i)} d(\vec{s}),
\]
where
\[
\Gamma(i) = \bigcup_{\substack{\alpha + \beta = i \\ 0 \le \alpha \le t(s-1) \\ 0 \le \beta \le t(r-1)}}  \leftsub{\beta}{\HH}(r) \cap \HH_{\alpha}(s).
\]
\end{enumerate}
Summation over an empty set is interpreted to result in $0$.
\end{theorem}
It will be technically convenient actually to work over a larger ring than $\Zp$.  %
%
%
Let $K = \Qp(\xi)$ be the unique unramified extension of degree $t(n+1)$ over $\Qp$, where $\xi$ is a primitive $(q^{n+1}-1)^{\text{\scriptsize th}}$ root of unity in $K$.  We set $\RR = \Zp[\xi]$ to be the ring of integers in $K$.  Then $\RR$ is a discrete valuation ring, $p \in \RR$ generates the maximal ideal, and $F = \RR/p\RR \cong \mathbb{F}_{q^{n+1}}$.%
%
%

Set $G = \GL(n+1, q)$.  If one fixes a basis of $V$ then there is %
%
%
a natural action of $G$ on the sets $\LL_{i}$, and in this way $\RR^{\LL_{i}}$ becomes an $\RR G$-permutation module.  As before, $A_{r,s}$ will denote both the matrix and the incidence map
\[
\RL{r} \to \RL{s}
\]
that sends an $r$-subspace to the (formal) sum of $s$-subspaces incident with it.  Since the action of $G$ preserves incidence, the $A_{r,s}$ are $\RR G$-module homomorphisms.  Clearly the $M_{i}(A_{r,s})$ and $N_{j}(A_{r,s})$ are $\RR G$-submodules.  We again have the $\RR G$-decomposition
\begin{equation*}
\RL{k} = \RR \mathbf{1} \oplus Y_{k},
\end{equation*}
where $\mathbf{1} = \sum_{x \in \LL_{k}} x$ and $Y_{k}$ is the kernel of the splitting map
\[
\sum_{x \in \LL_{k}} a_{x}x \mapsto \Bigl(\frac{1}{\abs{\LL_{k}}}\sum_{x \in \LL_{k}} a_{x}\Bigr)\mathbf{1},
\]
and all the $A_{r,s}$ respect this decomposition.  Reduction modulo $p$ %
%
%
induces a homomorphism of $FG$-permutation modules
\[
\FL{r} \to \FL{s},
\]
which we denote by $\overline{A_{r,s}}$.

Let us indicate how we will prove Theorem~\ref{thm:C}.  Suppose that we are able to find unimodular matrices $P$, $Q$, and $E$ such that 
\[
P A_{r,1} E^{-1} = D_{r,1}
\]
and
\[
E A_{1,s} Q^{-1} = D_{1,s}
\]
where the matrices on the right are diagonal.  Then these diagonal entries are the elementary divisors of the respective matrices $A_{r,1}$ and $A_{1,s}$.  Since then
\[
P A_{r,1} A_{1,s} Q^{-1} = D_{r,1} D_{1,s},
\]
we will then have obtained the elementary divisors of the product matrix (provided that we have detailed enough knowledge of the elementary divisors of the factor matrices).

In general is not possible to find such a matrix $E$ (\cite{rushanan} is a source of information on this topic).  Yet that is exactly what we will do.  The information that we need about the elementary divisors of $A_{r,1}$ and $A_{1,s}$ we obtain from \cite{chandler:sin:xiang:2006}.  
\begin{lemma}
\label{lem:D}
There exists a basis $\mathscr{B}$ of $\RL{1}$ that is simultaneously a left SNF basis for $A_{1,s}$ and a right SNF basis for $A_{r,1}$.
\end{lemma}
\begin{proof}
The group $G$ has a cyclic subgroup $S$ which is isomorphic to $F^\times$.
Since $\RR$ contains a primitive $\abs{S}^{\text{\scriptsize th}}$ root of unity, it follows that $K$
is a splitting field for $S$ and that the irreducible $K$-characters of $S$ 
take their values in $\RR$\@.
Let $\overline{S}$ denote the quotient of $S$ by the subgroup of scalar transformations.  Then $\overline{S}$ acts regularly on $\LL_{1}$, and $\abs{\overline{S}} = \abs{\LL_{1}}$ is a unit in $\RR$. Therefore, for each character $\chi$ of $\overline{S}$, the group ring $\RR \overline{S}$ contains an idempotent element $h_{\chi}$ that projects onto the (rank one) $\chi$-isotypic component of $\RL{1}$.  We thus obtain an $\RR$-basis $\mathscr{B} = \{v_{\chi} \, | \, \chi \in \Hom(\overline{S}, R^\times)\}$ of $\RL{1}$, where $v_{\chi} \in h_{\chi} \cdot \RL{1}$ such that $p \not \vert \,v_{\chi}$.

Now let us construct a left SNF basis for $A_{1,s}$, in the manner and notation described following the proof of Lemma~\ref{lem:A}.  Since each $F\overline{S}$-submodule of $\FL{1}$ is a direct sum of the isotypic components that it contains, we see that we can take each of the sets $\overline{\BB_{i}}$ to be a subset of $\overline{\mathscr{B}}$.  Suppose we lift $\overline{v_{\chi}} \in \overline{\BB_{i}}$ to an element $f \in M_{i}(A_{1,s})$.  Writing
\[
f = \sum_{\theta \in \Hom(\overline{S}, R^\times)} c_{\theta}v_{\theta},
\]
we see that $\overline{f} = \overline{c_{\chi}} \overline{v_{\chi}}$ and so $c_{\chi}$ must be a unit in $\RR$.  Since $M_{i}(A_{1,s})$ is an $\RR \overline{S}$-submodule, we have
that $h_{\chi} \cdot f = c_{\chi} v_{\chi}$ is also in $M_{i}(A_{1,s})$.  This proves that we may choose to lift $\overline{v_{\chi}}$ to $v_{\chi}$ in the construction, and that $\mathscr{B}$ is a left SNF basis for $A_{1,s}$.

An identical argument (lifting each $\overline{v_{\chi}}$ into some $N_{j}(A_{r,1})$) shows that $\mathscr{B}$ is a right SNF basis for $A_{r,1}$.
\end{proof}

It remains to show that the elementary divisor multiplicities are as stated in the theorem.  First we need a more precise description of the $FG$-submodule lattice of $\FL{1}$.  The facts that we need are as follows (see~\cite[Theorem A]{bardoe:sin})\@.  $\FL{1} = F\mathbf{1} \oplus \overline{Y_{1}}$ is a multiplicity-free $FG$-module, and the $FG$-composition factors of $\overline{Y_{1}}$ are in bijection with the set $\HH$.  The dimension over $F$ of the composition factor corresponding to the tuple $\vec{s}$ is $d(\vec{s})$.  Moreover, if we give $\HH$ the partial order
\[
(s_{0}, \dots , s_{t-1}) \le (s'_{0}, \dots , s'_{t-1}) \iff s_{i} \le s'_{i} \hbox{  for all i}
\]
then the $FG$-submodule lattice of $\overline{Y_{1}}$ is isomorphic to the lattice of order ideals of $\HH$, and the tuples contained in an order ideal correspond to the composition factors of the respective submodule.  Thus it is clear what is meant by %
%
%
the statement that a subquotient of $\overline{Y_{1}}$ \textit{determines} a subset of $\HH$.

\begin{remarks}\hfil
\begin{enumerate}
\item The field $k$ in~\cite{bardoe:sin} is actually an algebraic closure of $\Fq$, but (as observed in~\cite{chandler:sin:xiang:2006}) it follows from~\cite[Theorem A]{bardoe:sin} that all $kG$-submodules of $k^{\LL_{1}}$ are simply scalar extensions of $\Fq G$-modules, and therefore~\cite[Theorem A]{bardoe:sin} is also true over our field $F \cong \mathbb{F}_{q^{n+1}}$.  This observation also permits us to make use of certain results from \cite{chandler:sin:xiang:2006}, where the field is $\Fq$.
\item \label{item:inc} It should also be noted that the incidence relation considered in~\cite{bardoe:sin, sin:2004, chandler:sin:xiang:2006} is non-zero intersection (i.e., the complementary relation where two subspaces are incident if and only if their intersection is non-trivial).  If $A_{r,s}^{\prime}$ is the corresponding incidence matrix for non-zero intersection, then we have
\[
A_{r,s} = J - A_{r,s}^{\prime}.
\]
In particular,
\[
A_{r,s}|_{Y_{r}} = -A_{r,s}^{\prime}|_{Y_{r}}.
\]
Therefore the ($p$-adic) Smith normal forms of $A_{r,s}$ and $A_{r,s}^{\prime}$ can differ only with respect to where they map $\mathbf{1}$.  This accounts for the extra term appearing in the calculation of $p$-ranks in~\cite{bardoe:sin, sin:2004, chandler:sin:xiang:2006}.
\end{enumerate}
\end{remarks}

\begin{lemma}
\label{lem:C}
\leavevmode
\begin{enumerate}
\item \label{item:C1} The $FG$-module $\overline{M_{\alpha}(A_{1,s}|_{Y_{1}})}/\overline{M_{\alpha+1}(A_{1,s}|_{Y_{1}})}$ determines the subset $\HH_{\alpha}(s)$.
\item \label{item:C2} The $FG$-module $\overline{N_{\beta}(A_{r,1}|_{Y_{r}})}/\overline{N_{\beta-1}(A_{r,1}|_{Y_{r}})}$ determines the subset $\leftsub{\beta}{\HH}(r)$.
\end{enumerate}
\end{lemma}
\begin{proof}
Part~\eqref{item:C1} is the content of~\cite[Theorem 3.3]{chandler:sin:xiang:2006} (see Remarks above).  In order to prove~\eqref{item:C2}, first observe that for each $k$, $\LL_{k}$ is an orthonormal basis for a non-degenerate $G$-invariant symmetric bilinear form $\langle \cdot \, , \cdot \rangle_{k}$ on $\RL{k}$.  Use the induced form on $\FL{k}$ to identify each permutation module with its dual (contragredient) module, and observe that $\overline{A_{s,r}}$ is the dual map induced by $\overline{A_{r,s}}$.  Since the tuples $(s_{0}, \dots , s_{t-1})$ and $(n + 1 - s_{0}, \dots , n + 1 - s_{t-1})$ are determined by dual composition factors~\cite[Lemma 2.5(c)]{bardoe:sin}, part~\eqref{item:C2} will follow immediately if we can show the $FG$-module isomorphism
\[
\Bigl(\overline{N_{\beta}(A_{r,s}|_{Y_{r}})}/\overline{N_{\beta-1}(A_{r,s}|_{Y_{r}})}\Bigr)^{*} \cong \overline{M_{\beta}(A_{s,r}|_{Y_{s}})}/\overline{M_{\beta+1}(A_{s,r}|_{Y_{s}})}.
\]
It is sufficient to show that
\[
\overline{N_{\beta}(A_{r,s}|_{Y_{r}})}^{\perp} = \overline{M_{\beta + 1}(A_{s,r}|_{Y_{s}})}.
\]
We proceed by induction on $\beta$.  When $\beta = 0$, we have
\begin{align*}
\overline{N_{0}(A_{r,s}|_{Y_{r}})}^{\perp} &= \{\overline{y} \, \vert \, y \in Y_{s}, \, \langle(x)A_{r,s}, y\rangle_{s} \equiv 0 \pmod p \hbox{ for all $x \in Y_{r}$}\} \\
&= \{\overline{y} \, \vert \, y \in Y_{s}, \, \langle x, (y)A_{s,r}\rangle_{r} \equiv 0 \pmod p \hbox{ for all $x \in Y_{r}$}\} \\
&= \overline{M_{1}(A_{s,r}|_{Y_{s}})}.
\end{align*}
where the last equality follows from the non-degeneracy of the induced form on $\overline{Y_{r}}$.  Now assume $\beta > 0$.  It is easy to check that $\overline{M_{\beta + 1}(A_{s,r}|_{Y_{s}})} \subseteq \overline{N_{\beta}(A_{r,s}|_{Y_{r}})}^{\perp}$.  We then have 
\[
\overline{M_{\beta + 1}(A_{s,r}|_{Y_{s}})} \subseteq \overline{N_{\beta}(A_{r,s}|_{Y_{r}})}^{\perp} \subseteq \overline{N_{\beta - 1}(A_{r,s}|_{Y_{r}})}^{\perp} = \overline{M_{\beta}(A_{s,r}|_{Y_{s}})},
\]
with the equality by our induction hypothesis.  Since clearly $e_{\beta}(A_{s,r}|_{Y_{s}}) = e_{\beta}(A_{r,s}|_{Y_{r}})$, it now follows from Lemma~\ref{lem:A} and the above inclusions that $\overline{M_{\beta + 1}(A_{s,r}|_{Y_{s}})} = \overline{N_{\beta}(A_{r,s}|_{Y_{r}})}^{\perp}$.
\end{proof}
\begin{proof}[Proof of Theorem~\ref{thm:C}]
Fix an $FG$-composition series
\[
\{0\} \subseteq F\mathbf{1} = U_{0} \subseteq U_{1} \subseteq \cdots \subseteq \FL{1}.
\]
Starting with the $F$-basis $\{\overline{v_{1_{\overline{S}}}}\}$ %
%
%
of $U_{0}$, we can extend this using elements of $\overline{\mathscr{B}}$ to a basis of $U_{1}$.  Continuing in this fashion, we thus get the disjoint union 
\[
\mathscr{B} = \{v_{1_{\overline{S}}}\} \cup \DD_{1} \cup \cdots
\]%
%
%
where $\overline{\DD_{i}}$ are the elements of $\overline{\mathscr{B}}$ extending $U_{i-1}$ to $U_{i}$.  It is clear that each quotient $U_{i}/U_{i-1}$ ($i \ge 1$) is isomorphic as an $F\overline{S}$-module to the $F\overline{S}$-submodule of $\overline{Y_{1}}$ spanned by $\overline{\DD_{i}}$.  If the simple $FG$-module $U_{i}/U_{i-1}$ determines the tuple $\vec{s} \in \HH$, then we will say that %
%
%
each element of $\DD_{i}$ \textit{determines} the tuple $\vec{s}$.  This assignment of elements of $\mathscr{B}$ to tuples in $\HH$ is well-defined independent of the above composition series, as follows from the fact that the isomorphism type of an $F\overline{S}$-submodule of $\overline{Y_{1}}$ is completely determined by the characters it affords.

By Lemma~\ref{lem:C}, the tuple determined by $v_{\chi}$ belongs to $\HH_{\alpha}(s) \cap \leftsub{\beta}{\HH}(r)$ precisely when the following two conditions hold:
\begin{enumerate}
\item $\overline{v_{\chi}} \in \overline{M_{\alpha}(A_{1,s}|_{Y_{1}})}$ but $\overline{v_{\chi}} \notin \overline{M_{\alpha+1}(A_{1,s}|_{Y_{1}})}$
\item $\overline{v_{\chi}} \in \overline{N_{\beta}(A_{r,1}|_{Y_{r}})}$ but $\overline{v_{\chi}} \notin \overline{N_{\beta-1}(A_{r,1}|_{Y_{r}})} $.
\end{enumerate}
It immediately follows that 
\[
e_{i}(A_{r,1}A_{1,s}|_{Y_{r}}) = \sum_{\alpha + \beta = i} \sum_{\vec{s} \in \HH_{\alpha}(s) \cap \leftsub{\beta}{\HH}(r)} d(\vec{s}), \hbox{\quad for $i \ge 0$}.
\]
Since $\HH_{\alpha}(s) = \emptyset$ for $\alpha > t(s-1)$ and $\leftsub{\beta}{\HH}(r) = \emptyset$ for $\beta > t(r-1)$, we have
\[
e_{i}(A_{r,1}A_{1,s}|_{Y_{r}}) = 0, \hbox{\quad for $i > t(r+s-2)$}.
\]
We will use the $q$-binomial coefficients
\[
\begin{bmatrix}m\\ \ell\end{bmatrix}_{q}=\frac{(q^m-1)(q^{m-1}-1)\cdots(q^{m-\ell+1}-1)}{(q-1)(q^2-1)\cdots(q^\ell-1)}
\] 
for non-negative integers $m$ and $\ell$ with $m\geq\ell$.
Then
\[
(\mathbf{1})A_{r,1}A_{1,s} = q^{r+s}\qbinom{n}{r}\qbinom{n+1-s}{1}\mathbf{1},
\]
and we have $e_{t(r+s)}(A_{r,1}A_{1,s}) = 1$.  This completes the proof of Theorem~\ref{thm:C}.%
%
%
\end{proof}
\begin{remark}
Since $d(\vec{s})=0$ for $\vec{s} \in [n]^{t} \setminus \HH$, there is no effect on the numerical result of Theorem~\ref{thm:C} if we replace $\HH$ with $[n]^{t}$ in the notation preceding the statement of the theorem.
\end{remark}
\begin{proof}[Proof of Theorem~\ref{thm:B}]
Consider the situation when $n+1 = 4$ and $r = s = 2$ (so $A_{2,2} = A$ and $A_{1,2} = B$).  Replace $\HH$ with $[3]^{t}$ in the notation preceding Theorem~\ref{thm:C}.  Then it is easy to see that
\[
\HH_{\alpha}(2) = \{\vec{s} \in [3]^{t} \, | \, \vec{s} \hbox{ contains exactly $\alpha$ ones}\}
\]
and
\[
\leftsub{\beta}{\HH}(2) = \{\vec{s} \in [3]^{t} \, | \, \vec{s} \hbox{ contains exactly $\beta$ threes}\}.
\]
Hence
\begin{align*}
\Gamma(i) &= \bigcup_{\alpha+\beta=i}\bigl(\HH_{\alpha}(2) \cap \leftsub{\beta}{\HH}(2)\bigr) \\
&= \{\vec{s} \in [3]^{t} \, | \, \vec{s} \hbox{ contains exactly $t-i$ twos}\} \\
&= \HH(t-i).
\end{align*}
Therefore, for $0 \le i \le t$,
\[
e_{t-i}(B^{t}B) = \sum_{\vec{s} \in \HH(i)} d(\vec{s})
\]
and in view of equation~\eqref{eq:mult} we see that Theorem~\ref{thm:B} follows from Theorem~\ref{thm:C}.
\end{proof}

As mentioned in the introduction, the problem of computing the elementary divisors of $A_{r,s}$ in general is still very much unsolved.  The $p$-ranks of the incidence matrices $A_{r,s}$ were computed in~\cite{sin:2004}.  Observe that the $p$-rank of an integer matrix is just the multiplicity of $p^{0}$ as a $p$-adic elementary divisor.  %
%
%
We end the paper with the following easy corollary of Theorem~\ref{thm:C}.
\begin{corollary}
Notation is that of Theorem~\ref{thm:C}.  Let $e_{i}(A_{r,s})$ denote the multiplicity of $p^{i}$ as a $p$-adic elementary divisor %
%
%
of $A_{r,s}$.  Then, for $0 \le i < t$,
\[
e_{i}(A_{r,s}) = \sum_{\vec{s} \in \Gamma(i)} d(\vec{s}).
\]
\end{corollary}
\begin{proof}
Let $x \in \LL_{r}$.  Then 
\[
(x)A_{r,s} = \sum_{y \in \LL_{s}} a_{x,y}y,
\]
where
\begin{align*}
a_{x,y} &= \left\lvert\{z \in \LL_{1} \, \vert \, z \cap x = \{0\} \hbox{ and } z \cap y = \{0\}\}\right\rvert \\
&= \begin{cases}
   \qbinom{n+1}{1}-\qbinom{r}{1}-\qbinom{s}{1}, & \hbox{ if } x \cap y \neq \{0\} \\
   \qbinom{n+1}{1}-\qbinom{r}{1}-\qbinom{s}{1}+\qbinom{k}{1}, & \hbox{ if } \dim(x \cap y) = k \ge 1.
   \end{cases}
\end{align*}
Then $a_{x,y} \equiv -1 \pmod q$ when $x \cap y = \{0\}$ and $q$ divides $a_{x,y}$ otherwise.  Hence 
\[
A_{r,1}A_{1,s} \equiv -A_{r,s} \pmod {p^{t}}
\]
and the corollary now follows from Lemma~\ref{lem:B}.
\end{proof}
\section{Acknowledgements}
The second author received generous support through the Chat Yin Ho Memorial Scholarship, for which he wishes to thank the family and friends of Professor Ho. We also thank the Banff International Research Station, where discussion of this work began at a workshop in March, 2009.
\bibliographystyle{amsplain}

\end{document}